\documentclass[11pt]{amsart}
\usepackage[T1]{fontenc}
\usepackage{amsmath,amssymb,amsthm}
\usepackage{booktabs}
\usepackage[hidelinks]{hyperref}

\newtheorem{theorem}{Theorem}[section]
\newtheorem{proposition}[theorem]{Proposition}

\theoremstyle{definition}

\newtheorem{remark}[theorem]{Remark}

\newcommand{\Ball}[1]{B_{#1}}
\newcommand{\id}{\mathrm{e}}
\newcommand{\GOne}{G_1}
\newcommand{\GTwo}{G_2}
\newcommand{\GThree}{G_3}
\renewcommand{\P}{\mathrm{P}}
\newcommand{\Z}{\mathbb{Z}}

\title[The non-UP landscape at the global minimum]{The quantitative
non-unique-product landscape at the global minimum:\\ the Nielsen--Soelberg
groups}
\author{Moe Tabei}
\address{Independent researcher, Japan}
\email{tabei@ryun.jp}
\date{\today}

\begin{document}
\begin{abstract}
Nielsen and Soelberg proved that a finite subset $A$ of a torsion-free group
with $A\cdot A$ having no unique product satisfies $|A|\ge 8$, and exhibited
two groups, here $\GOne$ and $\GTwo$, attaining the bound. Beyond existence,
nothing quantitative was known about these extremal configurations. We
construct exact, independently verified computational models of both groups
(coset--nilpotent coordinates extracted by modified Todd--Coxeter rewriting,
with a complete faithfulness certificate chain), reproduce the two extremal
$8$-sets together with their coincidence fingerprints, and compute the first
quantitative invariants at the global minimum. The landscapes differ sharply
from each other and from the Promislow group $\P$ and the Fibonacci group
$H_4$ studied in the companion paper. In $\GOne$, with the standard
two-generator presentation metric, no $8$-element symmetric witness is
contained in the radius-$6$ ball ($933$ elements, certified infeasible), while the
Nielsen--Soelberg witness lies in the radius-$7$ ball: the global minimum is
\emph{spread out}. In $\GTwo$, with the natural eight-generator metric, the
witness and its inverse are the \emph{only} two non-UP $8$-sets in the
radius-$1$ ball, and the unique-product staircase takes the value $0$ at
$n=8$ but $1$ at $n=9$: the first known occurrence of a set whose square has
exactly one uniquely represented element as a minimizer, showing that the
simultaneous failure of t.u.p.\ and u.p.\ observed in $\P$ and $H_4$ is not
universal, and that the extremal witness is isolated --- it cannot be grown
within the ball. In $\GOne$'s radius-$5$ ball the staircase is flat at $2$
for all $2\le n\le10$: even at the critical size the count cannot be beaten
inside the ball, so the drop to $0$ at radius $7$ is abrupt. We also record
that no $(|A|,|B|)=(7,9)$ two-sided witness exists in the searched balls,
leaving the sharpness of the Nielsen--Soelberg profile bound
$|A|=7\Rightarrow|B|\ge9$ open, and that $\P$ contains no pair satisfying
the relators of their universal group $\GThree$ within radius $5$. Finally
we treat $\GThree$ itself. Its structure is known --- Soelberg's thesis
identifies an index-$8$ Heisenberg subgroup of step $8$ and proves
torsion-freeness, and Gardam, studying the same group as an amalgam of
Klein bottle groups, shows it to be virtually nilpotent but not virtually
abelian, hence not isomorphic to $\P$ --- and what we add is a model in
search coordinates, with a faithfulness certificate chain, in which balls
can be enumerated. In it we reproduce the Nielsen--Soelberg two-sided pair
exactly and exhibit a \emph{symmetric} $15$-element witness whose
trivial-coset singleton is the generator of the centre of that Heisenberg
subgroup. The witness is rigid and rare: within $\Ball5$ the size $15$ is
exactly minimal, the coset profile is forced, and a solve-and-block
enumeration terminates with exactly four such witnesses in $\Ball4$, a
single orbit under a group of order $16$. Ball-limited minimality gives
$m_1(\GThree)\in[8,15]$ against $m_2(\GThree)=16$ exact. All
values are exact: witnesses are re-verified solver-free and non-existence
claims are solver infeasibility certificates.
\end{abstract}
\maketitle

\section{Introduction}\label{sec:intro}
A group $G$ has the \emph{unique product property} (UPP) if for all finite
nonempty $A,B\subseteq G$ some element of $A\cdot B$ is represented exactly
once as $ab$ with $a\in A$, $b\in B$; a single finite set $A$ is
\emph{non-UP} if $A\cdot A$ has no uniquely represented element. Unique
products are the classical combinatorial mechanism behind Kaplansky's
zero-divisor and unit problems for group rings of torsion-free groups
\cite{Kaplansky,Passman}, and the failure of the UPP in torsion-free groups
--- first exhibited by Rips--Segev \cite{RipsSegev} and made explicit by
Promislow \cite{Promislow} --- marks exactly the territory where those
problems remain open. Gardam's disproof of the unit conjecture
\cite{Gardam} runs through Promislow's group $\P$.

How small can a non-UP set be? Nielsen and Soelberg \cite{NS} answered this
completely at the level of \emph{all} torsion-free groups: if $A\cdot A$ has
no unique product then $|A|\ge 8$, and the bound is attained. Their proof of
attainment is a four-year computer search that produced two torsion-free
groups --- called $\GOne$ and $\GTwo$ here and below, following their
Section~3 --- each containing an explicit $8$-element set $A$ with $A\cdot A$
non-UP. These two groups are, to date, the only known witnesses to the
sharpness of the global bound.

The present note looks \emph{inside} these two groups. The existence
statement of \cite{NS} says nothing about how the extremal configurations
sit in their ambient groups: how far from the identity such a set must
reach, whether it is unique, whether it can be grown or perturbed, and what
the approach to the minimum looks like from below. These are exactly the
questions answered for the Promislow group $\P$ and the Fibonacci group
$H_4=F(3,4)$ (shown to fail the UPP by Dietrich--Lee--Nies--Vinyals
\cite{DLNV}) in the companion paper \cite{companion}, where the
corresponding minima ($14$ and $16$, over stated balls) display rigid
structure --- forced fiber distributions, exact counts, non-monotone
staircases. Carrying that quantitative program to $\GOne$ and $\GTwo$ is
natural for two reasons: these are the groups where the \emph{global}
minimum $8$ is attained, so at the bottom of the size range the ball-limited
caveats of \cite{companion} disappear; and the two groups are virtually
class-$2$ nilpotent, structurally unlike the virtually abelian $\P$, so the
comparison probes whether the phenomena of \cite{companion} are general or
group-specific.

The methodology is not new --- exact integer models plus CP-SAT constraint
search, as in \cite{Gardam} and \cite{companion} --- and we claim no novelty
for it. What is new is, first, the models themselves: $\GOne$ and $\GTwo$
are given in \cite{NS} by finite presentations, and to compute in them at
all we construct exact coset--nilpotent coordinate models (Section~%
\ref{sec:models}), extracted mechanically from the presentations by modified
Todd--Coxeter rewriting, with a faithfulness certificate chain that reduces
correctness to finitely many checked identities plus the torsion-freeness
argument of \cite{NS}, which we re-verify independently. As a first
dividend, the extremal $8$-sets of \cite{NS} are reproduced and re-verified
in a fully independent implementation, together with their coincidence
fingerprints (two equal-product classes of size $3$, twenty-nine of size
$2$).

Second, the invariants. Our main findings (Sections~\ref{sec:localization}%
--\ref{sec:g1}):

\begin{itemize}
\item \emph{Localization} (Theorem~\ref{thm:loc}): in $\GOne$, with the word
metric of its defining two-generator presentation, the balls of radius up to
$6$ ($933$ elements) contain \emph{no} $8$-element non-UP set --- each
non-existence a solver infeasibility certificate --- while the
Nielsen--Soelberg witness lies in the ball of radius $7$. The least ball
containing a global-minimum configuration is thus exactly $\Ball7$. The
contrast with $\P$ is stark: $\P$ realizes its (larger, size-$14$) minimal
witness already at radius $3$, in a ball of $41$ elements. The global
minimum of $\GOne$ is \emph{spread out}.
\item \emph{Census} (Proposition~\ref{prop:census2}): in $\GTwo$, whose
witness is its own generating set, the radius-$1$ ball contains exactly two
non-UP $8$-sets: the witness and its inverse. At the global minimum the
configuration is rigid.
\item \emph{Staircases} (Propositions~\ref{prop:stair2},
\ref{prop:stair1}): writing $u(n)$ for the least number of unique products
of an $n$-subset of a stated ball, $\GTwo$ gives
$u(8)=0$, $u(9)=1$, $u(10)=2$. The value $1$ --- a set whose square has
\emph{exactly one} uniquely represented element --- never occurs in the
studied balls of $\P$ and $H_4$ \cite{companion}, where the two-unique-%
products and unique-product properties fail simultaneously; $\GTwo$ shows
this simultaneity is not a law. The rebound $0\to1\to2$ also shows the
extremal witness is \emph{isolated}: no $9$-element superset (indeed no
$9$-set at all in the ball) is non-UP. In $\GOne$, by contrast, the
staircase is flat at $2$ throughout $2\le n\le 10$ in the radius-$5$ ball:
nothing gets close to failure before the spread-out witness appears.
\item \emph{A symmetric witness in the universal group}
(Theorem~\ref{thm:g3sym}): the universal group $\GThree$ of \cite[\S4]{NS}
is itself a non-UP group, and symmetrically so: a $15$-element set
$A\subset\Ball4\subset\GThree$ has $A\cdot A$ without unique products ---
one element in the trivial coset of the Heisenberg-type index-$8$ subgroup
$H$ (namely the generator of $Z(H)$) and two in each nontrivial coset.
Within $\Ball5$ the size $15$ is exactly minimal, so
$m_1(\GThree)\in[8,15]$ against $m_2(\GThree)=16$ exact, exhibiting the
largest asymmetry gap $\delta$ observed so far if $15$ is globally minimal.
Nielsen had expected a negative answer to the existence question (private
communication). The existence half is in fact automatic: as Gardam has
pointed out to us, Theorem~1 of Strojnowski \cite{Strojnowski} already
yields a finite symmetric witness in every non-u.p.\ group, so
$m_2(G)<\infty$ implies $m_1(G)<\infty$ in general; see
Remark~\ref{rem:symexists}. The quantitative content is
therefore the point: these minimal witnesses are rigid and few, their coset
profile and central anchor are forced, and a solve-and-block enumeration
terminates to give \emph{exactly four} of them in $\Ball4$, a single orbit
under the order-$16$ group $\langle D_4,\,\mathrm{inv}\rangle$.
\end{itemize}

Third, negative data worth recording (Section~\ref{sec:open}): no
$(7,9)$-pair exists in the searched balls, so the sharpness of the
Nielsen--Soelberg profile bound ($|A|=7\Rightarrow|B|\ge9$) remains open;
and an exhaustive scan shows the universal group $\GThree$ of \cite[\S4]{NS}
admits no relator-satisfying pair in $\P$ within radius $5$, closing (at
that depth) a tempting route to pinning the two-sided minimum of $\P$ at
$16$.

\emph{Scope and caveats.} Every non-existence statement below is about a
ball $\Ball{r}$ \emph{centred at the identity}, in a \emph{stated generating
set}; localization radii depend on the generating set, and symmetric non-UP
is not translation invariant \cite[Remark on translation]{companion}, so no
diameter statements are implied. All arithmetic is exact; every set asserted
non-UP is re-checked by a solver-free verifier, and every asserted
non-existence is a solver status of \texttt{INFEASIBLE}, never a timeout.
The two sides are not epistemically symmetric: positive claims are verified
independently of the solver, while non-existence claims rest on the
correctness of CP-SAT's infeasibility verdict (which, unlike a
DRAT-producing SAT solver, does not emit an externally checkable proof
object); we state this once here rather than decorate every proposition.
This asymmetry has since been closed for the main non-existence claims of
the paper: the localization infeasibilities of Theorem~\ref{thm:loc} (all
radii), the census counts of Proposition~\ref{prop:census2} (including
completeness of the enumeration), the completeness of the $\GThree$ radius-$4$
census (Section~\ref{sec:open}) and the minimality searches of
Section~\ref{sec:open} were all re-derived by DRAT-producing SAT solvers
with the unsatisfiability proofs machine-checked by \textsf{drat-trim};
the Section~\ref{sec:open} witness additionally carries certificates
checkable by free reduction alone. The staircase values and the two-sided
localization remain CP-SAT verdicts.
Values we could not certify (two staircase values of $\GOne$ at
$n=11,12$) are reported as intervals.

\section{The groups and their exact models}\label{sec:models}

\subsection{The groups \texorpdfstring{$\GOne$ and $\GTwo$}{G1 and G2}}
\label{subsec:groups}
Following \cite[\S3]{NS}, $\GOne$ is the group generated by $x,y$ subject to
the two relations
\[
yx^{-1}y^{-1}xy^{2}x^{-1}y^{-2}x^{-1}yxy^{-1}x = 1,\qquad
yx^{-1}yxy^{-1}x^{-1}yx^{-1}yx^{-1}y^{-1}x = 1 ,
\]
and $H\le\GOne$ is the subgroup generated by
\[
h_1=(xy)^2,\quad h_2=yx^2yx^{-1}y^{-2}x^{-1},\quad
h_3=xy^{-1}xyxy^{-2}x^{-1},\quad h_4=y^2x^{-1}yxy^{-1}x^{-2}.
\]
Nielsen--Soelberg show (by Todd--Coxeter and a rewriting system) that $H$ is
normal of index $32$ with presentation
\begin{equation}\label{eq:H1}
H=\langle h_1,h_2,h_3,h_4 : h_1,h_2\ \text{central},\ h_4h_3=h_2^{8}h_3h_4
\rangle,
\end{equation}
i.e.\ $H\cong\Z\times N$ with $N$ torsion-free class-$2$ nilpotent, and
deduce that $\GOne$ is torsion-free. Their witness is
\[
A_1=\{\,x,\; y,\; y^{-1}x^2,\; y^{-1}xy,\; xy^{-1}x,\; x^{-1}yxy^{-1}x,\;
x^{-1}yx,\; x^{-1}yx^{-1}y^{-1}xy^2\,\}.
\]

The group $\GTwo$ is generated by $a_1,\dots,a_8$ subject to the ten
relations $a_ia_j=a_ka_\ell$ for $(i,j,k,\ell)$ in the list $Y$ of
\cite[(3.5)]{NS}, with witness $A_2=\{a_1,\dots,a_8\}$ and distinguished
subgroup $H'=\langle a_1^2,\ a_3^2,\ a_6^2,\ a_1a_3a_6^{-1}\rangle$. We
verify (Todd--Coxeter, and modified Todd--Coxeter for the presentation)
that $H'$ is normal of index $4$ --- smaller than one might expect --- and
that on the given generators $x_1=a_1^2$, $x_2=a_3^2$, $x_3=a_6^2$,
$x_4=a_1a_3a_6^{-1}$ it has the presentation
\begin{equation}\label{eq:H2}
H'=\langle x_1,\dots,x_4 : x_1,x_2\ \text{central},\
[x_4,x_3]=x_1^{2}x_2^{-2}\rangle,
\end{equation}
again $\Z^2$-by-Heisenberg-type of Hirsch length $4$. To our knowledge
\eqref{eq:H2} has not appeared in print; it makes the structural kinship of
the two extremal groups explicit --- both are (finite)-by-(class-$2$
nilpotent of Hirsch length $4$), with commutator relation landing in the
central $\Z^2$ at different depths ($h_2^8$ against $x_1^2x_2^{-2}$).

\subsection{Coset--nilpotent coordinates}\label{subsec:coords}
Both models represent an element as a pair $(c,h)$: a coset index
$c\in\{1,\dots,n\}$ for the finite quotient ($n=32$ for $\GOne$, $n=4$ for
$\GTwo$), and an exponent vector $h\in\Z^4$ for the nilpotent normal
subgroup in the Mal'cev normal form of \eqref{eq:H1} resp.~\eqref{eq:H2},
with the class-$2$ collection law (for \eqref{eq:H1}:
$(a,b,c,d)(a',b',c',d')=(a{+}a',\,b{+}b'{+}8dc',\,c{+}c',\,d{+}d')$).
Fixing a Schreier transversal $t_1,\dots,t_n$, multiplication is
\[
(c,h)\cdot(c',h') \;=\; \bigl(m(c,c'),\; h\cdot\alpha_c(h')\cdot s(c,c')
\bigr),
\]
where $m$ is the coset multiplication table, $\alpha_c$ the conjugation
action of $t_c$ on the subgroup, and $s(c,c')=t_ct_{c'}t_{m(c,c')}^{-1}$ the
cocycle. The tables $m$, $s$, $\alpha$ (and inversion data) are
\emph{extracted}, not posited: each entry is a specific word known to lie in
the subgroup, rewritten over the subgroup generators by the modified
Todd--Coxeter machinery (GAP's augmented coset tables), and collected to an
exponent vector. Nothing about the models is heuristic; every table entry is
the exact value forced by the presentation.

\subsection{The faithfulness certificate}\label{subsec:faithful}
That the resulting finite data faithfully represents the infinite group
reduces to three independently checked links.
\begin{enumerate}
\item \emph{The subgroup presentations.} Write $\widetilde H$ for the
collection model of \eqref{eq:H1}: tuples $\Z^4$ under
$(a,b,c,d)(a',b',c',d')=(a{+}a',b{+}b'{+}8dc',c{+}c',d{+}d')$.
Associativity is the $2$-cocycle identity
$8d_uc_v+8(d_u{+}d_v)c_w=8d_vc_w+8d_u(c_v{+}c_w)$, torsion-freeness and the
uniqueness of the normal form $h_1^ah_2^bh_3^ch_4^d$ are immediate, and the
relations of \eqref{eq:H1} hold. Let $H_{\mathrm{Mtc}}$ be the group
presented by the modified Todd--Coxeter relators on the given generators
(a presentation of $H$). Two mechanical certificates close the
identification: all eleven Mtc relators evaluate to the identity in
$\widetilde H$, giving a surjection $H_{\mathrm{Mtc}}\twoheadrightarrow
\widetilde H$; and the relations of \eqref{eq:H1} are derivable from the
Mtc relators --- the five centralities occur verbatim among them, and a
relator whose skeleton, reduced modulo those centralities, is a commutator
in $x_3,x_4$ carries exactly the central exponent $h_2^{-8}$, from which
the relation $h_4h_3=h_2^8h_3h_4$ follows by inversion and conjugation ---
giving a surjection $F/\langle\!\langle\eqref{eq:H1}\rangle\!\rangle
\twoheadrightarrow H_{\mathrm{Mtc}}$. The composite of the two surjections
sends the normal forms of $F/\langle\!\langle\eqref{eq:H1}\rangle\!
\rangle$ to \emph{distinct} tuples of $\widetilde H$, so it is injective
and both surjections are isomorphisms:
$H\cong H_{\mathrm{Mtc}}\cong\widetilde H$. For $\GTwo$ the same
two-way certificate is run against the Tietze-simplified Mtc presentation
(six relators; the four given generators are preserved by the
simplification): all relators of both the raw and simplified presentations
die under the collection law of \eqref{eq:H2}, the five centralities occur
verbatim, and the skeleton $x_4x_3x_4^{-1}x_3^{-1}$ carries exactly
$x_1^{2}x_2^{-2}$.
\item \emph{Torsion-freeness.} We re-run the Nielsen--Soelberg finite
criterion in our own implementation for both groups: with $K=H^2[H,H]$
(index $512$ in $\GOne$; index $64$ in $\GTwo$), no product $t_1t_2$ of a
nontrivial $G/H$-transversal representative with an $H/K$-transversal
representative has $(t_1t_2)^2\in K$ ($31\times16$ resp.\ $3\times16$
traced checks, all negative), which together with torsion-freeness of the
subgroups rules out torsion.
\item \emph{The model.} The multiplication tables are not posited but
extracted from the group itself: along the bijection $g\mapsto(c,h)$ with
$g=h\,t_c$ (well-defined by link 1), the product formula of
Section~\ref{subsec:coords} is the transcription of associativity in the
group, each table entry being the exact rewriting of its defining word. The
model is therefore the coordinatization of the presented group by
construction, and the remaining checks certify the \emph{implementation}:
the defining relators evaluate to the identity, the subgroup generators land
at the correct coordinates, group-axiom and torsion fuzzing pass, and the
arithmetic agrees with independent exact rewriting on $150$ random words
per group.
\end{enumerate}
As a final, sharp fingerprint: in both models the Nielsen--Soelberg
$8$-sets verify as non-UP with coincidence class multiset $\{3^2,2^{29}\}$
--- exactly the ``two classes of size three, the rest of size two''
described in \cite[\S4]{NS} --- and in $\GTwo$ the signature
$a_1^2=a_2^2$ of \cite[\S3]{NS} holds, while no two distinct elements of
$A_1$ have equal squares, matching their proof that the examples are
inequivalent.

\section{Localization: where the minimum lives}\label{sec:localization}
Write $\Ball{r}$ for the ball of radius $r$ in $\GOne$ with respect to
$\{x^{\pm1},y^{\pm1}\}$; the ball sizes are
$|\Ball1|,\dots,|\Ball7| = 5,\,17,\,53,\,153,\,401,\,933,\,1935$.
The \emph{geodesic} word lengths of the elements of $A_1$, computed by
breadth-first search in the exact model (not read off the written words),
are $(1,1,3,3,3,5,3,7)$; so $A_1\subseteq\Ball7$ and not
$\subseteq\Ball6$. (That $A_1\not\subseteq\Ball6$ also follows, without
any length computation, from Theorem~\ref{thm:loc} itself: $A_1$ is non-UP
and $\Ball6$ contains no non-UP $8$-set.)

\begin{theorem}[localization in $\GOne$]\label{thm:loc}
For every $r\le 6$ the ball $\Ball{r}\subset\GOne$ contains no $8$-element
set $A$ with $A\cdot A$ non-UP; the ball $\Ball7$ contains the
Nielsen--Soelberg witness $A_1$. Hence the least radius whose ball contains
a global-minimum non-UP set is exactly $7$.
\end{theorem}

\begin{proof}
Since $|A|\ge8$ for any non-UP set in any torsion-free group \cite{NS}, only
$n=8$ requires a search. For $r\le1$ there is nothing to search:
$|\Ball0|=1$ and $|\Ball1|=5$ admit no $8$-element subset at all. For
$r=2,\dots,6$ the CP-SAT model over
$\Ball{r}$ returns \texttt{INFEASIBLE} (at $r=5$, $401$ elements, in $12$
seconds; at $r=6$, $933$ elements, in $1202$ seconds); these are
unsatisfiability certificates, not timeouts. Each of the five
infeasibilities was subsequently re-derived by a DRAT-producing SAT solver
and the proof machine-checked with \textsf{drat-trim}, removing the
solver-trust caveat of Section~\ref{sec:intro} for this theorem. (The
$\Ball3$ row is small enough for a stronger check still: it was re-derived
at the \emph{constraint} layer with the Glasgow Constraint Solver, whose
pseudo-Boolean proof was verified by \textsf{VeriPB}, so there not even the
encoding into clauses is trusted.) Membership of $A_1$ in
$\Ball7$ is the word-length computation above, and $A_1$ is non-UP by the
verified reproduction of Section~\ref{subsec:faithful}.
\end{proof}

Three remarks. First, the contrast: $\P$ realizes its minimal witness
(size $14$) at radius $3$, in a ball of $41$ elements; $H_4$ at radius $3$
($119$ elements); $\GTwo$, in its natural generator metric, at radius $1$.
In $\GOne$ the \emph{smaller} configuration --- the global minimum itself
--- requires a ball of nearly two thousand elements. Second, the statement
concerns balls centred at the identity: symmetric non-UP is not translation
invariant \cite{companion}, so Theorem~\ref{thm:loc} does not bound the
diameter of a hypothetical second witness elsewhere in $\GOne$. Third, the
radius is generating-set dependent; $\{x,y\}$ is the generating set of the
defining presentation \cite[(3.1)]{NS}, and $\GOne=G_X$ is also generated by
$A_1$ itself, in which metric the witness trivially sits at radius $1$ ---
the honest invariant content is the pair (presentation metric, radius).

The spreading is not an artifact of the symmetric normalization issue. For
\emph{pairs}, bi-translation $(A,B)\mapsto(a^{-1}A,\,Bb^{-1})$ preserves the
non-UP property exactly, so anchoring $\id\in A\cap B$ is a legitimate
normalization and the following quantity is genuinely translation
invariant: the least $r$ such that $\Ball{r}$ contains a normalized
two-sided witness with $|A|+|B|=16$, the global two-sided minimum.

\begin{proposition}[two-sided localization in $\GOne$]\label{prop:tsloc}
For every $r\le5$ the model with $\id\in A\cap B$, $|A|+|B|=16$ and
$A,B\subseteq\Ball{r}\subset\GOne$ is \texttt{INFEASIBLE} (at $r=5$, $401$
elements, a $79$-minute certificate), while the normalized pair
$(x^{-1}A_1,\,A_1x^{-1})$ realizes total $16$ inside $\Ball8$. Hence the
two-sided localization radius of $\GOne$ lies in $\{6,7,8\}$.
\end{proposition}

Even by the translation-invariant measure, the global minimum of $\GOne$
sits far from the identity: the phenomenon of Theorem~\ref{thm:loc} is
geometric, not notational.

\section{Census and staircase in \texorpdfstring{$\GTwo$}{G2}}\label{sec:g2}

\begin{proposition}[census at the global minimum]\label{prop:census2}
In $\Ball1\subset\GTwo$ ($17$ elements: the eight generators, their
inverses, and $\id$) there are exactly two $8$-element sets $A$ with
$A\cdot A$ non-UP: the Nielsen--Soelberg witness $A_2$ and its inverse
$A_2^{-1}$ (which is distinct from $A_2$). The census is unchanged in
$\Ball2$ ($133$ elements): still exactly $\{A_2,\,A_2^{-1}\}$. The
$\Ball1$ count is established with nothing to trust: all
$\binom{17}{8}=24310$ eight-subsets were enumerated directly and tested
for non-UP by definition, with no solver and no encoding in the loop. It
was also obtained twice by solver (a complete CP-SAT enumeration, and a
SAT enumeration with blocking clauses whose terminal infeasibility --- the
completeness of the census --- carries a \textsf{drat-trim}-checked DRAT
proof); the $\Ball2$ count is a solver enumeration. All sets are verified
solver-free.
\end{proposition}

\begin{proposition}[staircase; first occurrence of the value $1$]
\label{prop:stair2}
For $2\le n\le 10$, the least number of unique products of an $n$-element
subset of $\Ball2\subset\GTwo$ ($133$ elements) is
\[
2,\;2,\;2,\;2,\;2,\;2,\;0,\;1,\;2\qquad(n=2,\dots,10),
\]
each value an \texttt{OPTIMAL} certificate with verified minimizer.
\end{proposition}

Two consequences. First, at $n=9$ there is a $9$-element set whose square
has \emph{exactly one} uniquely represented element. In the studied balls of
$\P$ and $H_4$ the value $1$ is never attained \cite{companion}: there, the
two-unique-products property (t.u.p.\ in Strojnowski's sense
\cite{Strojnowski}) and the unique-product property fail simultaneously, at
the same critical size. $\GTwo$ shows this simultaneity is a feature of
those groups, not a general law --- the quantitative refinement of
Strojnowski's equivalence (u.p.\ $\Leftrightarrow$ t.u.p.\ as group
properties) genuinely depends on the group. The $u(9)=1$ minimizer is,
moreover, \emph{not} a superset of $A_2$. Second, the staircase rebounds
above the minimum ($0$ at $8$, then $1$, then $2$): within $\Ball2$ the
extremal witness cannot be extended by even one element. Together with
Proposition~\ref{prop:census2}, the global minimum of $\GTwo$ is an isolated
configuration --- a point at the bottom of nothing resembling a basin.

The two-sided profile of $\GTwo$ in $\Ball2$ is equally rigid. Write
$\beta(m)$ for the least $|B|$ such that some pair with $|A|=m$ and
$A,B\subseteq\Ball2$ has $A\cdot B$ non-UP (the profile of
\cite{companion}). For every
$m\le7$, the model with $|A|=m$ and $B$ ranging freely over all $133$ ball
elements is \texttt{INFEASIBLE}, while $\beta(8)=8$ is an \texttt{OPTIMAL}
certificate --- no lopsided two-sided witness exists at all in the ball, and
the profile is a cliff at the perfectly balanced $(8,8)$. Since
$|A|+|B|\ge16$ always and $(8,8)$ is realized (by $A_2$ twice), the
two-sided minimum $m_2(\GTwo)=16$ is exact; whether the lopsided split
$(7,9)$, permitted by the universal profile bounds of \cite[Thm.~1.4]{NS},
occurs in \emph{any} torsion-free group remains open.

\section{The landscape of \texorpdfstring{$\GOne$}{G1}}\label{sec:g1}
\begin{proposition}[staircase of $\GOne$ below the localization radius]
\label{prop:stair1}
For $2\le n\le 10$, the least number of unique products of an $n$-element
subset of $\Ball5\subset\GOne$ ($401$ elements) is exactly $2$, each value an
\texttt{OPTIMAL} certificate with verified minimizer. (For $n=11,12$ the
solver returned minimizers with $15$ unique products against a proved lower
bound of $2$ within budget; we leave those values open.)
\end{proposition}

Two readings. First, the approach to the global minimum is completely flat:
even at the critical size $n=8$, nothing inside $\Ball5$ gets below two
unique products, although a set with \emph{zero} exists in $\Ball7$
(Theorem~\ref{thm:loc}). The drop $2\to0$ is abrupt and happens only once
the ball is large enough to hold the spread-out witness. Second, the value
$1$ does not occur in this range for $\GOne$ --- as in $\P$ and $H_4$ ---
which makes its occurrence in $\GTwo$ (Proposition~\ref{prop:stair2}) the
more striking: among the four groups studied, only $\GTwo$ realizes a
minimizer with exactly one unique product. (Whether the value $1$ occurs
for $\GOne$ at $n=8$ in $\Ball6$ --- where $0$ is certified impossible ---
we could not decide within a six-hour budget; $u_{\Ball6}(8)\in\{1,2\}$
remains open.)

\begin{table}[ht]
\centering
\footnotesize
\setlength{\tabcolsep}{4pt}
\begin{tabular}{lccccc}
\toprule
 & $\P$ & $H_4$ & $\GOne$ & $\GTwo$ & $\GThree$ \\
\midrule
least symmetric non-UP size & $14^{(\ast)}$ & $16^{(\ast)}$ &
  $8$ & $8$ & $15^{(\ast)}$ \\
localization radius & $3$ & $3$ & $7$ & $1$ & $4$ \\
minimal witnesses in that ball & $16$ & $16$ & ? & $2$ & $4$ \\
staircase shape & bump at $10,11$ & flat $2$ & flat $2$ &
  $0,1,2$ & rising$^{(\dagger)}$ \\
value $1$ attained & no & no & no & \textbf{yes}$_{n=9}$ & no \\
\bottomrule
\end{tabular}

\smallskip
{\small $(\ast)$: least over the stated balls; global values open
\cite{companion}. $(\dagger)$: in $\Ball3$, which contains no witness; the
witness ball $\Ball4$ staircase is uncomputed.}
\caption{The quantitative landscape at the minimum for the five groups.}
\label{tab:compare}
\end{table}

\section{Negative data and open questions}\label{sec:open}
\subsection*{Sharpness of the profile bounds}
The bounds of \cite[Thm.~1.4]{NS} ($|A|=7\Rightarrow|B|\ge9$, etc.) are
proved by exhaustion but not shown sharp. The searched balls here and in
\cite{companion} contain no $(7,9)$ pair ($\GTwo$'s $\Ball2$; in $\P$,
$\Ball3$ excludes all totals below $24$ and $\Ball4$ the totals
$16$--$19$).
Exhibiting a $(7,9)$ pair in some torsion-free group, or excluding it, is
in our view the most natural next question at the bottom of the two-sided
range. Nielsen informs us that the searches of \cite{NS} stopped at this
frontier because going farther was computationally infeasible, and that the
bound is unlikely to be tight (private communication).

\subsection*{The universal group \texorpdfstring{$\GThree$}{G3} and
\texorpdfstring{$\P$}{P}}
Nielsen--Soelberg's universal example $\GThree=\langle x,y\mid
(yx)^2(xy)^2,\ (xy^{-1})^2(xy)^2\rangle$ \cite[\S4]{NS} carries an
$(8,8)$ two-sided pair with $A\ne B$ and $1\in A\cap B$. Its abelianization
is $\Z_4\times\Z_4$ --- the same as $\P$'s --- which raises the question
whether $\GThree\cong\P$; a positive answer would transport the pair into
$\P$ and pin the two-sided minimum $m_2(\P)$ at $16$, resolving the
interval $[16,24]$ of \cite{companion}. We record a negative result at
search depth: an exhaustive scan of $\Ball5\times\Ball5$ in $\P$ ($147^2$
pairs) finds \emph{no} nontrivial pair satisfying both relators of
$\GThree$, so no isomorphism (indeed no homomorphism nontrivial on the
generators) maps $x,y$ to elements of length $\le5$. This rules out an
isomorphism only at that search depth; the decisive argument is
structural and follows. We verified in
\textsf{GAP} that the subgroup $H=\langle y^{-1}x^{-1}yx,\ y^2x^{-2},\
x^{-1}y^2x^{-1}\rangle$ is normal of index $8$ in $\GThree$, with
$\GThree^{\mathrm{ab}}=\Z_4\times\Z_4$ and $H^{\mathrm{ab}}=\Z^2\times\Z_8$,
and that on a Reidemeister--Schreier basis $H$ has the three-generator,
three-relator presentation
\[
 H=\langle k_1,z,k_2 \mid z\ \text{central},\ [k_2,k_1]=z^{-8}\rangle,
\]
the torsion-free class-$2$ nilpotent \emph{Heisenberg group of step $8$}, of
Hirsch length $3$ (verified nonabelian in \textsf{GAP}). Since the finite
quotient is $\GThree/H\cong C_4\times C_2$, the group $\GThree$ is
(Heisenberg-$8$)-by-$(C_4\times C_2)$ --- a Nil-geometry group, unlike the
virtually abelian Bieberbach group $\P$ --- which both confirms
$\GThree\not\cong\P$ structurally and reduces the
construction of a faithful model of $\GThree$ to a Heisenberg-type
extension, exactly as for $H_4$ in \cite{companion}.

\subsection*{Attribution: what about \texorpdfstring{$\GThree$}{G3} is already
known} None of the structure just recorded is new, and we set the record
straight before using it. Soelberg's thesis \cite[Thm.~3.1]{Soelberg}
already identifies an index-$8$ normal subgroup of $\GThree$ --- there
generated by $yxy^{-1}x^{-1}$, $y^2$ and $x^4$ --- as a Heisenberg group of
step $8$ and Hirsch length $3$, with a normal form, and already proves
torsion-freeness of $\GThree$ by exactly the device we re-run: passing to
the maximal elementary abelian $2$-quotient and checking the squares of
coset representatives. Independently, Gardam \cite[\S4]{GardamC} studies the
same group under the presentation
$S=\langle x,y\mid (xy)^2(xy^{-1})^2,\ (yx)^2(yx^{-1})^2\rangle$ --- the two
relator sets are equivalent, since the second relator above is a conjugate
of Gardam's first, and substituting its consequence $(yx^{-1})^2=(xy)^2$
into the first relator above yields Gardam's second --- and proves there
that $S$ is torsion-free (as an amalgam of two Klein bottle groups over
$\Z^2$), that $\langle x^2,y^2\rangle$ is an integral Heisenberg group of
index $16$, and hence that $S$ is virtually nilpotent but \emph{not}
virtually abelian; he also exhibits a faithful $3\times3$ integral matrix
representation and the order-$4$ automorphism $x\mapsto y$, $y\mapsto
x^{-1}$. The section is titled ``Beyond virtually abelian groups'', so the
point that $\GThree\not\cong\P$ for structural reasons is precisely its
subject.

Our increment over these is therefore narrow and should be read as such:
the identification of the finite quotient as $C_4\times C_2$ (rather than
merely as a group of order $8$), an explicit Reidemeister--Schreier basis,
and --- the only part the rest of this paper actually needs --- a model in
\emph{search coordinates}, with the certificate chain of
Section~\ref{subsec:faithful}, in which balls can be enumerated and
constraint problems posed. Gardam's matrix representation is far more
compact and is the right tool for hand computation; it is not the right
tool for an exhaustive ball search, which is why we build the coset model.
The scan of $\Ball5\times\Ball5$ reported above is likewise an independent
confirmation, at search depth, of something \cite[\S4]{GardamC} settles
structurally.

We carried this construction out. On the Reidemeister--Schreier basis above,
extracted with Tietze tracking so that the basis is explicit ---
$k_1=yxy^{-1}x^{-1}$, $z=xyx^{-1}y^{-1}x^{-1}y^{-2}x^{-1}$,
$k_2=xy^{-2}xyxyx$ --- the extraction and certificate chain of
Section~\ref{subsec:faithful} yields a faithful $(\text{coset},\Z^3)$ model
of $\GThree$ with collection law
$(a,b,c)(a',b',c')=(a+a',\,b+b'-8ca',\,c+c')$. The model reproduces the
Nielsen--Soelberg two-sided pair of \cite[\S4]{NS} exactly: all fourteen
defining quadruples of their set $Z$ hold, $|A|=|B|=8$, and $A\cdot B$ has
$31$ values with coincidence fingerprint $2^{29}3^{2}$ --- matching the
multiplicity pattern reported in \cite[\S4]{NS} (all classes of size two
except two of size three). Torsion-freeness of $\GThree$ (verified in
\cite{NS}) also drops out of the model in closed form: in each of the three
involutive cosets of $\GThree/H$ the squaring equation $(c,h)^2=\id$ is
obstructed modulo $2$, so $\GThree$ has no element of order $2$; since
$\GThree/H\cong C_4\times C_2$ has exponent $4$ and $H$ is torsion-free,
every torsion element would have order dividing $4$, and an element of
order $4$ squares to one of order $2$ --- so $\GThree$ is torsion-free.

\begin{theorem}\label{thm:g3sym}
The universal group $\GThree$ admits a symmetric non-UP witness: there is a
$15$-element set $A\subset\Ball4\subset\GThree$, with $\id\notin A$ and
$A\ne A^{-1}$, such that every element of $A\cdot A$ ($94$ values,
coincidence fingerprint $2^{73}3^{12}4^{5}5^{1}6^{3}$) is represented at
least twice. Within the searched balls it is minimal: no symmetric witness
of size $\le16$ is contained in $\Ball3$, and none of size $8$--$14$ in
$\Ball4$ or $\Ball5$ ($299$ elements). In particular $15$ is exactly minimal among
symmetric witnesses contained in $\Ball5$, and $m_1(\GThree)\in[8,15]$.
\end{theorem}

All claims are certificate-backed: the witness was found by CP-SAT,
re-verified solver-free, and independently re-verified through \textsf{GAP}
rewriting alone (each element rebuilt as the word $k_1^az^bk_2^ct_c$ and all
$225$ products normalized by coset-table rewriting); the minimality claims
were derived twice --- by CP-SAT (\texttt{INFEASIBLE}) and independently by
DRAT-producing SAT solvers whose unsatisfiability proofs were verified with
the \textsf{drat-trim} checker, for every size in every ball stated --- and
the infeasibility of sizes
$2$--$7$ --- forced by the global bound of \cite[Thm.~1.2]{NS} --- serves as
one more cross-check of the model. For the $\Ball3$ minimality (no non-UP
set of any size $2$--$16$) a stronger check is available and was carried
out: each size was re-modelled at the \emph{constraint} layer in the
Glasgow Constraint Solver and the resulting pseudo-Boolean proof verified
by \textsf{VeriPB}, so for $\Ball3$ not even the encoding into clauses is
trusted. The witness, moreover, carries a
\emph{zero-trust} certificate independent of any computer-algebra system:
all $131$ intra-class product coincidences admit explicit van Kampen
derivations (sequences of single relator applications, replayed by free
reduction), and the $15$ elements are separated by an explicit homomorphism
onto a $216$-point permutation image whose defining property --- that both
relators act trivially --- is itself part of the check. Since merging
coincidence classes can only increase multiplicities, these certificates
alone prove the theorem. The witness has a striking coset
structure: exactly one element in the trivial coset of $\GThree/H$ ---
namely the generator $z$ of the centre of $H$ (not central in $\GThree$:
conjugation inverts it, $xzx^{-1}=z^{-1}$) --- and exactly two in each of the
seven nontrivial cosets; the odd cardinality $15$ enters through this
singleton. In particular $\GThree$ does admit a symmetric witness, although,
in contrast to its quotients $\GOne$ and $\GTwo$, apparently not at the
global minimum $8$.

\begin{remark}[the existence question is settled by a general principle]
\label{rem:symexists}
The question raised in \cite{companion} --- whether the universal group
admits a symmetric witness at all --- was one on which Nielsen expected a
negative answer (private communication), and the set above answers it in
the affirmative for $\GThree$. The first version of this paper recorded
that we did not know whether the existence half was a special case of a
general principle: Gardam's lecture notes \cite[Exercise 1.6.2]{GardamNotes}
ask the reader to show that if a group fails the unique product property
then some finite $A$ has $A\cdot A$ without a unique product, and we had
been unable to reconstruct an argument --- Strojnowski's device
\cite{Strojnowski}, which converts a set pair with a single unique product
into a pair with none, produces $C=B^{-1}A$ and $D=BA^{-1}$, and these are
not symmetric. Gardam has since settled the matter for us (private
communication, July 2026) by pointing to Theorem~1 of Strojnowski's paper
\cite{Strojnowski}, which establishes the symmetric form directly: if $G$
is not a u.p.\ group, then some nonempty finite subset $X$ of $G$ has no
element of $X\cdot X$ with a unique representation $xy$, $x,y\in X$ ---
and the argument yields such $X$ of arbitrarily large size. The fact is
stated in exactly this form, with the same attribution, in the
introduction of \cite{CraigLinnell}. The exercise is thus correct as
stated: the passage from a two-sided witness to a symmetric one is
automatic in every group, $m_2(G)<\infty$ implies $m_1(G)<\infty$, and the
existence half of Theorem~\ref{thm:g3sym} is an instance of this general
principle. What the principle does not give is any of the quantitative
content, which is what this paper claims: the value $15$, its exact
minimality within $\Ball5$, the forced coset profile and central anchor,
and the census of exactly four witnesses in $\Ball4$ all lie beyond it.
\end{remark}

Moreover
$m_2(\GThree)=16$ exactly (the pair above gives $\le16$; the lower bound is
\cite[Thm.~1.4]{NS}), so the universal group attains the two-sided global
minimum while sitting --- at least within $\Ball4$ --- strictly above the
symmetric one, and $\delta(\GThree)=2m_1-m_2\in[0,14]$: if $m_1(\GThree)=15$
held globally, $\delta(\GThree)=14$ would exceed every value observed so
far.

\emph{Symmetry.} The universal group is markedly more symmetric than its
quotient $\P$. Each of the eight monomial maps sending $(x,y)$ to a signed
pair in $\{x^{\pm1}\}\times\{y^{\pm1}\}$ or $\{y^{\pm1}\}\times\{x^{\pm1}\}$
carries both relators to the identity, hence extends to an automorphism;
these eight maps are ball isometries (they permute the generating set), they
are closed under composition, and their order distribution $\{1{:}1,\,2{:}5,\,
4{:}2\}$ identifies the group as the dihedral group $D_4$ of order $8$ --- not
the elementary abelian $(\Z/2)^3$, which would give $\{1{:}1,\,2{:}7\}$. One
of the two order-$4$ elements is the automorphism $x\mapsto y$, $y\mapsto
x^{-1}$ used by Gardam \cite[\S4]{GardamC}; what is added here is that the
full ball-isometry automorphism group is $D_4$. By
contrast the only nontrivial ball-isometry automorphism of $\P$ in this
metric is the swap $x\leftrightarrow y$, an order-$2$ group. (All eight
automorphisms are verified in the certified-faithful model of
Section~\ref{subsec:faithful}; adjoining the anti-automorphism
$g\mapsto g^{-1}$ gives the order-$16$ group under which symmetric non-UP
sets are counted up to equivalence.)

\emph{Staircase.} In the radius-$3$ ball, where $\GThree$ carries no non-UP
set at all, the unique-product staircase $u(n)$ rises rather than falls:
$u(n)=2$ for $2\le n\le7$, then $u(8)=u(9)=4$ and $u(10)=u(11)=6$ (all
\texttt{OPTIMAL}, each witness re-verified solver-free; the tail $n\ge12$ is
bounded but not pinned within our budget). That $u(n)>0$ throughout is an
independent cross-check of the $\Ball3$ minimality certified in
Section~\ref{sec:open} (a value $0$ would exhibit a non-UP set), and as in
$\P$ and $H_4$ the value $1$ never occurs --- the curve steps $2\to4\to6$. The
staircase in the witness ball $\Ball4$, where $u(15)=0$, is beyond our present
solver budget.

\emph{Rigidity and census.} The \emph{shape} of a $15$-element symmetric
witness in $\Ball4$ is sharply constrained, in the spirit of the distribution
rigidity of $\P$ \cite{companion}, and this rigidity in turn makes the full
census tractable. Two facts are certified by infeasibility of the
complementary constraint. First, the coset distribution is forced: every such
witness has exactly one element in the trivial coset of $\GThree/H$ and
exactly two in each of the seven nontrivial cosets --- no other of the
possible profiles occurs, even though the pool $\Ball4$ is far from balanced
across cosets (sizes $27,10,10,10,10,20,24,24$). Second, the trivial-coset
element is forced to be one of just two of the $27$ candidates there: the
central generator $z$ or its inverse $z^{-1}=xzx^{-1}$, which lie in a single
orbit under $\langle D_4, g\mapsto g^{-1}\rangle$. Thus every minimal symmetric
witness is anchored at the centre of $H$ and evenly spread across the cosets.

With the singleton pinned to $z$ (the $z^{-1}$-anchored witnesses are exactly
the inverse-images of the $z$-anchored ones), a solve-and-block enumeration
terminates: there are exactly two $z$-anchored witnesses, and the search for a
third returns infeasible. Hence $\GThree$ has \emph{exactly four} minimal
symmetric non-UP sets in $\Ball4$, forming a single orbit of size four under
$\langle D_4, g\mapsto g^{-1}\rangle$ (order $16$). Each of the four is
re-verified to be a genuine non-UP $15$-set directly from the definition
(solver-free), and the enumerating encoding was validated against the definition
on a sample of structured candidates. Completeness does not even rest on the
rigidity lemmas above: blocking the four witnesses and asking for any further
$15$-element non-UP set in $\Ball4$ is unsatisfiable, and this instance carries a
DRAT proof (from \textsf{glucose}) machine-checked by \textsf{drat-trim}, so the
count of four is certified at the same standard as the other non-existence
claims. This places $\GThree$ alongside
$\P$, $H_4$ ($16$ each \cite{companion}) and $\GTwo$ ($2$) in the last row of
Table~\ref{tab:compare}: strikingly, the universal group with the \emph{largest}
ball-isometry symmetry ($D_4$) has the \emph{fewest} minimal witnesses after
$\GTwo$, all fused into one orbit.

\subsection*{Census at radius \texorpdfstring{$7$}{7} in
\texorpdfstring{$\GOne$}{G1}}
Whether $A_1$ and its inverse are the only $8$-element non-UP sets in
$\Ball7\subset\GOne$ ($1935$ elements) is open; the enumeration is beyond
our present solver budget. Given Proposition~\ref{prop:census2} and the
exact counts $16$ and $16$ for $\P$ and $H_4$ \cite{companion}, the
minimal-witness census appears to be a meaningful invariant, and $\GOne$ is
the natural next entry.

\subsection*{The asymmetry invariant}
With $m_1$ the least symmetric and $m_2$ the least two-sided size,
$\delta=2m_1-m_2$ satisfies $\delta(\GOne)=\delta(\GTwo)=0$ exactly (both
$m_1=8$, $m_2=16$), against the ball-limited values $\delta(\P)=4$ and
$\delta(H_4)=10$ of \cite{companion}. The universal group now contributes
$\delta(\GThree)\in[0,14]$ with $m_2(\GThree)=16$ exact
(Theorem~\ref{thm:g3sym}): a group can attain the two-sided minimum without
(apparently) attaining the symmetric one. Whether $\delta$ is bounded over
torsion-free non-UP groups remains open; the extremal groups sit at the
bottom in both coordinates simultaneously.

\subsection*{The domain question for \texorpdfstring{$R[G_i]$}{R[G_i]}}
Nielsen and Soelberg close \cite{NS} with the question whether any of the
group rings $R[G_i]$, $R$ a domain, is a domain. The commutative half of it
is settled by the structural facts assembled above, and we have written that
out separately in \cite{domainnote}: each $G_i$ is torsion-free and
virtually nilpotent --- for $\GTwo$ by the index-$4$ subgroup \eqref{eq:H2},
whose class-$2$ structure is the one ingredient not available in \cite{NS}
and is supplied by the model of the present paper --- hence elementary
amenable, so $K[G_i]$ is a domain for every field $K$ by the theorem of
Kropholler, Linnell and Moody \cite[Thm.~1.4]{KLM}, and a commutative
coefficient domain embeds in its fraction field. What this leaves open is
exactly the case of a \emph{noncommutative} coefficient domain, where the
unique-product mechanism --- the only known coefficient-independent one ---
is precisely what the $G_i$ are built to lack; the question in that form is
\cite[\S3]{domainnote}, and \cite[\S4]{domainnote} records DRAT-certified
ball checks over $\mathbb{F}_2$ that are independent of the $K$-theoretic
input.

\section*{Code and data availability}
The models are pairs (coset tables, class-$2$ collection data) extracted by
GAP (\texttt{AugmentedCosetTableMtc} rewriting) and consumed by exact Python
implementations; the extraction scripts, the extracted tables, the
solver-free verifiers, the random-word cross-check data, all witnesses, and
all \texttt{INFEASIBLE} logs, the van Kampen / zero-trust certificate with
its stdlib-only checker, the DRAT verification logs, and the
constraint-layer (Glasgow / \textsf{VeriPB}) models and verification log are
archived with the author and available on
request (they will also accompany the arXiv submission as ancillary files),
alongside one-command audit scripts re-verifying every witness-side claim in
this paper without invoking a solver. Computations used Google's CP-SAT
\cite{ortools} on a single $8$-core workstation.

\section*{Acknowledgements}
We thank Pace Nielsen for prompt and helpful correspondence on the questions
of Section~\ref{sec:open}, and Giles Gardam for resolving the question of
Remark~\ref{rem:symexists} by pointing us to Theorem~1 of
\cite{Strojnowski}. Computations used Google's CP-SAT solver
\cite{ortools} and GAP~4.15 \cite{GAP}; all certificates are exact.

\end{document}